\pdfoutput=1
\documentclass[psamsfonts]{amsart}

\usepackage{article}

\begin{document}
\author{Paige Randall North}
\thanks{This material is based upon work supported by the Air Force Office of Scientific Research under award number FA9550-16-1-0212.}
\title{Identity types and weak factorization systems in Cauchy complete categories}

\begin{abstract}
It has been known that categorical interpretations of dependent type theory with $\Sigma$- and $\Id$-types induce weak factorization systems. When one has a weak factorization system $(\mathcal L, \mathcal R)$ on a category $\C$ in hand, it is then natural to ask whether or not $(\mathcal L, \mathcal R)$ harbors an interpretation of dependent type theory with $\Sigma$- and $\Id$- (and possibly $\Pi$-) types. Using the framework of \emph{display map categories} to phrase this question more precisely, one would ask whether or not there exists a class $\mathcal D$ of morphisms of $\C$ such that the retract closure of $\mathcal D$ is the class $ \mathcal R$ and the pair $(\C, \mathcal D)$ forms a display map category modeling $\Sigma$- and $\Id$- (and possibly $\Pi$-) types. In this paper, we show, with the hypothesis that $\C$ is Cauchy complete, that there exists such a class $\mathcal D$ if and only if $(\C, \mathcal R)$ \emph{itself} forms a display map category modeling $\Sigma$- and $\Id$- (and possibly $\Pi$-) types. Thus, we reduce the search space of our original question from a potentially proper class to a singleton.
\end{abstract}

\maketitle

\tableofcontents

\section{Introduction}

In this paper, we study categorical interpretations of dependent type theory \cite{NPS00}. It has long been known that 
dependent type theory with $\Sigma$- and $\Id$-types can be interpreted in certain weak factorization systems \cite{War08, AW09} and that such interpretations {induce} weak factorization systems \cite{GG08}. Thus, a search for such interpretations could comprise two steps: first, identify a weak factorization system $(\mathcal L, \mathcal R)$ on a category $\C$, and second, decide if $(\mathcal L, \mathcal R)$ harbors an interpretation of dependent type theory with $\Sigma$- and $\Id$-types. Since we are interested in the connection between dependent type theory and familiar weak factorization systems of homotopy theory, we are interested primarily in the second step.

The present paper is the first in a series of papers which develop a theorem for recognizing whether a given weak factorization system $(\mathcal L, \mathcal R)$ harbors a model of dependent type theory with $\Sigma$- and $\Id$- (and possibly $\Pi$-) types. The content of this series can already be found in the author's thesis \cite{Nor17}. This paper is a streamlined account of the second chapter of \cite{Nor17}.

In this and following papers, we choose \emph{display map categories} from the various categorical frameworks which can interpret dependent type theory. This is because the data of a display map category, which consists of a category $\C$ and a class of maps of $\C$, is directly comparable to the data underlying a weak factorization system, which consists of a category $\C$ and two classes of maps of $\C$ (each of which determine the other). 

Not only do we choose the simplest categorical framework for interpreting dependent type theory, but we have also chosen the simplest variant of \emph{weak factorization system} (compared to, for example, \emph{algebraic} weak factorization systems). We make these choices in order to reveal the most fundamental connection between these two notions.

In this paper, we study display map categories $(\C, \D)$ which model $\Sigma$- and $\Id$- (and possibly $\Pi$-) types. As mentioned above, such a structure generates a weak factorization system $(^\boxslash \D, \overline \D)$ on the category $\C$ where $\overline \D$ is the retract closure of $\D$.  In this framework, for any weak factorization system $(\mathcal L, \mathcal R)$ on a category $\C$, our original question, 
\begin{quest}
Does $(\mathcal L, \mathcal R)$ harbor an interpretation of dependent type theory with $\Sigma$- and $\Id$- (and possibly $\Pi$-) types?
\end{quest} 
\noindent can be phrased more precisely as the following question.
 \begin{quest} \label{secondquestion}
 Does there exist a subclass $\mathcal D \subseteq \mathcal R$ such that $\mathcal R$ is the retract closure of $\mathcal D$ and $(\C, \mathcal D)$ is a display map category which models $\Sigma$- and $\Id$- (and possibly $\Pi$-) types?
 \end{quest}
\noindent The original contribution of this paper is the following theorem.

\begin{mainthm}
Consider a Cauchy complete category $\C$ and a display map category $(\C,  \D)$ which models $\Sigma$-types and functorial $\Id$-types. Then $(\C, \overline \D)$ is again a display map category modeling $\Sigma$- and functorial $\Id$-types. If $(\C,  \D)$ also models $\Pi$-types, then $(\C, \overline \D)$ also models $\Pi$-types.
\end{mainthm}

With this theorem, Question \ref{secondquestion} is equivalent to the following question.

 \begin{quest}
Is $(\C, \mathcal R)$ a display map category with $\Sigma$- and $\Id$- (and possibly $\Pi$-) types?
 \end{quest}
 Thus, to decide whether or not $(\mathcal L, \mathcal R)$ harbors an interpretation, we do not have to analyze the pair $(\C, \D)$ 
for all classes $\D$ whose retract closure is $\mathcal R$, which likely constitute a proper class. Rather, we only need to analyze the one pair $(\C, \mathcal R)$. Indeed, this will be the subject of the next paper of this series.

\section{Display map categories and quantifiers}
\label{sec:dmc}

In this section, we fix definitions of \emph{display map categories} and of \emph{$\Sigma$-} and \emph{$\Pi$-types} in display map categories. 

\begin{defn}
A \emph{display map category} $(\C, \D)$ consists of a category $\C$ with a terminal object and a class $\D$ of morphisms of $\C$ such that:
\begin{enumerate}
\item $\D$ contains every isomorphism;
\item $\D$ contains every morphism whose codomain is a terminal object;
\item every pullback of every morphism of $\D$ exists; and
\item $\D$ is stable under pullback.
\end{enumerate}
We call the elements of $\D$ \emph{display maps}.
\end{defn}

In such a display map category, the objects of $\C$ are meant to represent contexts, and the morphisms of $\C$ represent context morphisms. A morphism $p: E \to B$ of $\mathcal D$ represents a type family $E$ dependent on $B$. The empty context is represented by the terminal object of $\C$, so condition (2) says that every object of $\C$ may also be viewed as a type dependent on the empty context. The pullback of a morphism $p$ of $\mathcal D$ along a morphism $f$ of $\C$ represents the substitution of $f$ into the type family $p$.

\begin{defn}\label{sigmatypes}
A {display map category $(\C, \D)$} \emph{models $\Sigma$-types} if $\mathcal D$ is closed under composition.  We call a composition $gf$ of display maps a \emph{$\Sigma$-type} and sometimes denote it by $\Sigma_g f$.
\end{defn}

\begin{defn}\label{pilem}
A display map category $(\C, \D)$
\emph{models $\Pi$-types} if for every pair of composable display maps $g:W \to X$ and $f:X \to Y$, there exists a display map $\Pi_f g$ with codomain $Y$ and the universal property 
\[ \C / Y(y, \Pi_f g) \cong \C / X(f^* y, g)\]
natural in $y$. The term \emph{$\Pi$-type} will refer to such a display map $\Pi_f g$.
\end{defn}

\begin{rem}
The definitions in this section are relatively standard in the literature.
A display map category which models $\Sigma$-types coincides with Joyal's notion of \emph{clan}, and a display map category which models $\Sigma$- and $\Pi$-types coincides with his notion of \emph{$\pi$-clan} \cite{Joy17}. 
A \emph{class $\D$ of displays} in $\C$, in the sense of Taylor \cite{Tay99}, where all identities and morphisms to the terminal object are in $\D$ is a display map category $(\C, \D)$. His strong sums and dependent products coincide with our $\Sigma$-types and $\Pi$-types. Criteria (1)-(4) of Shulman's definition of \emph{type-theoretic fibration category} constitute a display map category with $\Sigma$- and $\Pi$-types \cite{Shu15}.
\end{rem}

\section{Identity types in display map categories}
Now we define $\Id$-types in a display map category. This definition is more convoluted and less standard than the definitions of $\Sigma$- and $\Pi$-types, but in Section \ref{sec:justifyid} below, we justify this choice of definition by comparing it with others.

First we fix some notation.

\begin{notn}
For a class $\M$ of morphisms of a category $\C$, let $^\boxslash \M$ denote the class of morphisms of $\C$ which have the left lifting property against $\M$. Similarly, let $\M^\boxslash$ denote the class of morphisms of $\C$ which have the right lifting property against $\M$.
\end{notn}

\begin{defn} \label{Idtypes}
Consider a {display map category} ($\C, \D$) which models $\Sigma$-types. We say that it \emph{models $\Id$-types} if for every $f:X \to Y$ in $\mathcal D$, 
\begin{enumerate}
\item the diagonal $\Delta_f: f \to f \times f$ in the slice $\C / Y$ has a factorization $\Delta_f = \epsilon_f r_f$ in $ \C /Y$ 
\begin{equation} \diagram
X \ar[dr]_f \ar[r]^-{r_f} & \Id(f) \ar[d]^{\iota_f} \ar[r]^-{\epsilon_f} & X \times_Y X \ar[dl]^{f \times f} \\
& Y 
\enddiagram \label{eq:idfact}
\end{equation}
 \end{enumerate}
 such that
 \begin{enumerate}[resume]
 \item $\epsilon_f$ is in $ \mathcal D$ and
 \item for every morphism $\alpha: A \to X$ in $\C$, the pullback $ \alpha^* r_f $, as shown below, is in $^\boxslash\mathcal{D}$ for $i = 0,1$.
\begin{equation}
\diagram 
&\alpha^*\Id(f) \ar[rrr] \ar[dd]|!{"2,1";"2,4"}\hole &&&{\Id(f)} \ar[dd]^{\pi_i \epsilon_f} \\
A \ar[rrr] \ar@{=}[dr] \ar[ur]^{\alpha^* r_f} &&& X \ar@{=}[dr] \ar[ur]^{r_f} & \\
& A \ar[rrr]^\alpha && & X
\enddiagram 
\end{equation}
\end{enumerate}

We will call the morphism $\iota_f: \Id(f) \to Y$ in Diagram (\ref{eq:idfact}) the \emph{$\Id$-type of $f$} in $ \C /Y$. 
\end{defn}

Note that since $(\C, \D)$ models $\Sigma$-types in this definition, $\D$ is closed under composition and stable under pullback. Thus, for any $f \in \D$, $f \times f$ is in $\D$ as it is the composition of a pullback of $f$ with $f$, and $\iota_f$ is in $\D$ since it is the composition of $\epsilon_f$ and $f \times f$.

\begin{defn}
Consider a {display map category} ($\C, \D$) which models $\Sigma$-types and $\Id$-types. For any object $Y$, let $\D / Y$ denote the full subcategory of the slice category $\C / Y$ spanned by those objects which are display maps. Let $\mathcal F$ denote the category $\bullet \to \bullet \to \bullet$ so that $(\D / Y)^{\mathcal F}$ is the category of composable pairs of morphisms of $\D / Y$.

We say that $(\C, \D)$ \emph{functorially models $\Id$-types} if for each object $Y$ of $\C$, there is a functor $ \D / Y \to (\D / Y)^{\mathcal F}$ which provides the factorization required by part (1) of Definition \ref{Idtypes} above. 
\end{defn}

\subsection{Comparison with other identity types}
\label{sec:justifyid}
Our definition of identity types is not completely standard, so we pause here to compare it to others in the literature. The unconcerned reader can safely skip this section.

The identity types given above correspond to one of several ways in which identity types may be defined in the syntax of dependent type theory. All the variants that we will discuss here start with the standard formation and introduction rules.
\[
\inferrule
  { \Gamma \vdash a,b : A}
  { \Gamma \vdash \Id_A(a,b)}
   \hspace{3em}
\inferrule
{ \Gamma \vdash a:A}
{ \Gamma \vdash r_a: \Id_A(a,a)}
\]
These must respect substitution: that is, we have the following meta-theoretic rules which Warren \cite{War08} calls \emph{coherence rules}.
\[
\inferrule
  {\Gamma, x: T, \Theta \vdash a,b : A \\  \Gamma,  \Delta \vdash t : T}
  { \Gamma, \Delta, \Theta[t/x] \vdash \Id_A(a,b)[t/x] = \Id_{A[t/x]}(a[t/x],b[t/x])}
\]\[
\inferrule
  {\Gamma, x: T, \Theta \vdash a: A \\  \Gamma,  \Delta \vdash t : T}
  { \Gamma, \Delta, \Theta[t/x] \vdash r_a[t/x] = r_{a[t/x]} : \Id_A(a,b)[t/x] }
\]
Then the elimination and computation rules may be given in one of several ways:
\begin{enumerate}
\item The (non-parametrized) elimination and computation rules \`a la Martin-L\"of
\[
\inferrule
  { \Gamma,a,b:A, p: \Id_A(a,b) \vdash E(a,b,p) \\
   \Gamma,a :A \vdash e(a) : E(a,a,r_a)
   }
  {\Gamma,a,b:A, p: \Id_A(a,b) \vdash j(e,a,b,p):E(a,b,p) \\ \Gamma,a :A \vdash e(a) = j(e,a,a,r_a) : E(a,a,r_a) }
\]
together with the appropriate coherence rule.
\end{enumerate}

In the absence of $\Pi$-types, these rules are not strong enough to prove many important properties of the identity type. The following two variants of the elimination and computation rules build some of the flexibility that $\Pi$-types provide directly into the identity types:
\begin{enumerate}[resume]
\item The parametrized elimination and computation rules \`a la Martin-L\"of 
\[
\inferrule
  {\Gamma,a,b:A, p: \Id_A(a,b), \Theta(a,b,p) \vdash E(a,b,p) \\
   \Gamma,a :A , \Theta(a,a,r_a) \vdash e(a) : E(a,a,r_a)
   }
  {\Gamma,a,b:A, p: \Id_A(a,b) , \Theta(a,b,p)\vdash j(e,a,b,p):E(a,b,p) \\ \Gamma,a :A, \Theta(a,a,r_a) \vdash e(a) = j(e,a,a,r_a) : E(a,a,r_a) }
\]
together with the appropriate coherence rule;

\item The elimination and computation rules \`a la Paulin-Mohring
\[
\inferrule
  { \Gamma \vdash a: A \\ \Gamma, b:A, p: \Id_A(a,b) \vdash E(b,p) \\
   \Gamma \vdash e : E(a,r_a) 
   }
  {\Gamma,b:A, p: \Id_A(a,b) \vdash j(e,b,p):E(b,p) \\ \Gamma \vdash e = j(e,a,r_a) : E(a,r_a) }
\]
\[
\inferrule
  { \Gamma \vdash a: A \\ \Gamma, b:A, p: \Id_A(b,a) \vdash E(b,p) \\
   \Gamma \vdash e : E(a,r_a) 
   }
  {\Gamma,b:A, p: \Id_A(b,a) \vdash j(e,b,p):E(b,p) \\ \Gamma \vdash e = j(e,a,r_a) : E(a,r_a) }
\]
together with the appropriate coherence rule.
\end{enumerate}
\begin{rem}
One might also consider parametrized elimination and computation rules \`a la Paulin-Mohring, by combining the variants (2) and (3).
\[
\inferrule
  { \Gamma \vdash a: A  \\ \Gamma, b:A, p: \Id_A(a,b),\Theta(b,p) \vdash E(b,p) \\
   \Gamma,\Theta(a,r_a) \vdash e : E(a,r_a) 
   }
  {\Gamma,b:A, p: \Id_A(a,b) , \Theta(b,p) \vdash j(e,b,p):E(b,p) \\ \Gamma , \Theta(a,r_a) \vdash e = j(e,a,r_a) : E(a,r_a) }\]
 However, we do not find it necessary to consider such strong $\Id$-types.
\end{rem}

In the presence of $\Pi$-types, these three variants (1)-(3) of the rules are all equivalent.  The fact that (1) is equivalent to (3) was first shown by Martin Hofmann, and can be found in \cite[Addendum, pp.\ 142-143]{Str93}. We show below in Proposition \ref{equividtypes} that the two strengthened variants, (2) and (3), are equivalent in the absence of $\Pi$-types at least in our categorical interpretation with the hypothesis that $^\boxslash \D$ is stable under pullback along $\D$.

Now, we define interpretations of these three variants of the identity type in a display map category. Note that we only model the coherence rules \emph{weakly}, in the sense of \cite{LW15}. 

\begin{defn}\label{lotsofidtypes}
Consider a category with display maps $(\C, \D)$ which models $\Sigma$-types. It \emph{models the formation and introduction rule of $\Id$-types} if for every $d:A \to \Gamma$ in $\mathcal D$, the diagonal $\Delta_d: d \to d \times d$ has a factorization $\Delta_d = \epsilon_d r_d$ in the slice $ \C /\Gamma$ such that $\epsilon_d$ is in $ \mathcal D$.
\[ \diagram
A \ar[dr]_d \ar[r]^-{r_d} & \Id(d) \ar[d]^{\iota_d} \ar[r]^-{\epsilon_d} & A \times_\Gamma A \ar[dl]^{d \times d} \\
& \Gamma 
\enddiagram \] 
 
If $(\C, \D)$ {models the formation and introduction rules of $\Id$-types} and if for every display map $d: A \to \Gamma$ and every morphism $\sigma: \Delta \to \Gamma$, the pullback of $\sigma^*r_d$ (illustrated in the following diagram) is in $^\boxslash \D$, we say that $(\C, \D)$ \emph{models Martin-L\"of $\Id$-types}.
\begin{equation*}
\diagram 
&\sigma^*\Id(d) \ar[rrr] \ar[dd]|!{"2,1";"2,4"}\hole &&&{\Id(d)} \ar[dd]^{\iota_d} \\
\sigma^* A \ar[rrr] \ar[dr] \ar[ur]^{\sigma^* r_d} &&& A \ar[dr]^{d} \ar[ur]^{r_d} & \\
& \Delta \ar[rrr]^\sigma && & \Gamma
\enddiagram
\end{equation*}

If $(\C, \D)$ {models Martin-L\"of $\Id$-types} and if for every display map $d: A \to \Gamma$, every morphism $\alpha: \Delta \to \Gamma$, and every display map $\theta: \Theta \to \sigma^* \Id(d)$, we have that $\theta^* (\sigma^* r_d)$ is in $^\boxslash \D$, then we say that $(\C, \D)$ \emph{models parametrized Martin-L\"of $\Id$-types}.

Now suppose that $(\C, \D)$ models the formation and introduction rule of $\Id$-types. Suppose also that for all display maps $d: A \to \Gamma$, objects $\Delta$ of $\C$, morphisms $\sigma: \Delta \to A$, and $i \in \{0,1\}$, the pullback $\sigma^* r_d$ of $r_d$ shown below is in $^\boxslash \D$. Then we say that $(\C, \D)$ \emph{models Paulin-Mohring $\Id$-types}.
\begin{equation}\label{pmdiag}
\diagram 
&\sigma^*\Id(d) \ar[rrr] \ar[dd]|!{"2,1";"2,4"}\hole &&&{\Id(d)} \ar[dd]^{\pi_i \epsilon_d} \\
\Delta \ar[rrr] \ar@{=}[dr] \ar[ur]^{\sigma^* r_d} &&& A \ar@{=}[dr] \ar[ur]^{r_d} & \\
& \Delta \ar[rrr]^\sigma && & A
\enddiagram 
\end{equation}
\end{defn}

The parametrized Martin-L\"of $\Id$-types correspond to the \emph{strong $\Id$-types} of \cite{BG12}. The Paulin-Mohring $\Id$-types are what we just call identity types in Definition \ref{Idtypes} and in the rest of this paper.

\begin{prop}\label{equividtypes}
Consider a display map category $(\C, \D)$ which models $\Sigma$-types and the formation and introduction rules of $\Id$-types. Then 
\begin{enumerate}
\item if $(\C, \D)$ {models parametrized Martin-L\"of $\Id$-types}, it models Paulin-Mohring $\Id$-types, and 
\item if $(\C, \D)$ models Paulin-Mohring $\Id$-types and $^\boxslash \D$ is stable under pullback along $\D$, then it models parametrized Martin-L\"of $\Id$-types.
\end{enumerate}
\end{prop}
\begin{proof}
Suppose that $(\C, \D)$ {models parametrized Martin-L\"of $\Id$-types}. We need to verify that the map $\sigma^*r_d$ in Diagram (\ref{pmdiag}) is in $^\boxslash \D$. This follows from Lemma 2.4 of \cite{Shu15}. That $(\C, \D)$ {models Martin-L\"of $\Id$-types} corresponds to Shulman's conditions ($5'$) and ($6'$) which he shows entail condition (6). Then, the fact that $\sigma^*r_d$ is in $^\boxslash \D$ is an instance of (6). Note that Shulman's stated hypotheses are stronger than ours, though ours suffice to prove this result. In particular, he assumes that $^\boxslash \D$ is stable under pullback along $\D$, but our weaker hypothesis that the $\Id$-types are parametrized can be used instead.

Now suppose that $(\C, \D)$ models Paulin-Mohring $\Id$-types and $^\boxslash \D$ is stable under pullback along $\D$. Consider a morphism $\theta^* (\sigma^* r_d)$ as in the definition (\ref{lotsofidtypes}) of parametrized Martin-L\"of $\Id$-types. Since $\theta \in \D$ and $\sigma^* r_d \in {^\boxslash \D}$, we have that $\theta^* (\sigma^* r_d) \in {^\boxslash \D}$.
\end{proof}

Thus, modulo the hypothesis that $^\boxslash \D$ is stable under pullback along $\D$, the conditions that a display map category $(\C, \D)$ model parametrized Martin-L\"of $\Id$-types and that it model Paulin-Mohring $\Id$-types are equivalent. In the next paper of this series, we will show that if a display map category $(\C, \D)$ models $\Sigma$-types and the formation and introduction rules of $\Id$-types, then it models Paulin-Mohring $\Id$-types if and only if $^\boxslash \D$ is stable under pullback along $\D$ (which already appears as Theorem 3.5.2 of \cite{Nor17}). Thus this hypothesis that $^\boxslash \D$ is stable under pullback along $\D$ is not necessary, but this is not the focus of the present paper.

The conditions (1), (2), (5), (6) of Shulman's type-theoretic fibration categories \cite[Def.\ 2.1]{Shu15} constitute a display map category $(\C, \D)$ which models $\Sigma$-types and Paulin-Mohring $\Id$-types.
A \emph{tribe} in the sense of Joyal \cite{Joy17} is a display map category $(\C, \D)$ which models $\Sigma$-types and Paulin-Mohring $\Id$-types, given this equivalence between Paulin-Mohring $\Id$-types and the stability of $^\boxslash \D$ under pullback along $\D$.

In summary, the $\Id$-types that we consider, the Paulin-Mohring $\Id$-types, are comparable to other categorical $\Id$-types that have appeared in the literature, and they are an appropriate version to study in the absence of $\Pi$-types. In what follows, we will return to calling Paulin-Mohring $\Id$-types just $\Id$-types.

\section{Weak factorization systems from display map categories}
\label{sec:wfs}

In this section, we recall how any display map category $(\C,\D)$ with $\Sigma$-types and $\Id$-types generates a weak factorization system $({ ^\boxslash \D},  \overline \D )$ with a factorization $(\lambda, \rho)$ where $\overline \D$ is the retract closure of $\D$ and the image of $\rho$ lies in $\D$. This will give us a good enough handle on the relationship between $\D$ and $\overline \D$ to prove our main theorem in the following section, where we extend an interpretation of type theory in $(\C, \D)$ to one in $(\C, \overline \D)$.

In the following proposition, we construct this weak factorization system. The proof uses ideas from the proof of Theorem 10 of \cite{GG08}, where a weak factorization system is constructed in the syntactic category of a dependent type theory. A categorical version appears as Theorem 2.8 in \cite{Emm14}.

\begin{notn}\label{closured}
	Let $\overline \D$ denote $(^\boxslash \mathcal D)^{\boxslash}$.
\end{notn}

\begin{prop}[{\cite[Thm.~2.8]{Emm14}}] \label{idwfs}
Consider a display map category $(\C, \D)$ which models $\Sigma$-types and $\Id$-types. There exists a weak factorization system $({ ^\boxslash \mathcal D},  \overline \D)$ in $\C$ with a factorization $(\lambda, \rho)$ where the image of $\rho$ is contained in $\mathcal D$. Furthermore, if the $\Id$-types are functorial, then this weak factorization system is functorial.
\end{prop}

\begin{proof}
We just describe here the factorization since it will be used later. The full proof of this statement can be found in \cite{Emm14}.

The factorization is defined in the following way for any $f: X \to Y$ in $\C$. We have a factorization 
\[ \diagram
Y \ar[r]^-{r_Y} & \Id(Y)\ar[r]^-{\epsilon_Y} & Y \times Y \enddiagram \]
of the diagonal $\Delta: Y \to Y \times Y$. Now we define the factorization of $f$ to be
\[ X \xrightarrow{1 \times r_Y f} X \times_Y \Id(Y) \xrightarrow{ \pi_1 \epsilon_Y} Y \]
where the middle object is obtained in the following pullback.
\[  \diagram X \times_Y \Id(Y) \ar[r] \ar[d] \pullback & \Id(Y) \ar[d]^{\pi_0 \epsilon_Y} \\
X \ar[r]^f & Y \enddiagram \]

The left factor 
\[\lambda(f) := 1 \times  r_Y f: X \to X \times_Y \Id(Y)\]
 is obtained as the following pullback of $r_Y$.
\[ \diagram 
&f^*\Id(Y) \ar[rrr] \ar[dd]|!{"2,1";"2,4"}\hole &&&{\Id(Y)} \ar[dd]^{\pi_0 \epsilon_Y} \\
X \ar[rrr] \ar@{=}[dr] \ar[ur]^{f^* r_Y} &&& Y \ar@{=}[dr] \ar[ur]^{r_Y} & \\
& X \ar[rrr]^f && & Y
\enddiagram \]
Thus, it is in $^\boxslash \mathcal D$.

The right factor
\[\rho(f):= \pi_1 \epsilon_Y : X \times_Y \Id(Y) \to Y\]
 is in $\mathcal D$ because it is the composition of a pullback of $\epsilon_Y$ with a pullback of $X \to *$.
\end{proof}

The class $\overline \D$, which was defined to be $(^\boxslash \D)^\boxslash$ in Notation \ref{closured}, is the retract closure of $\mathcal D$, justifying its notation.

\begin{lem}[{\cite[Prop.~14.1.8]{MP12}}] \label{basicliftingproperties}
Consider a display map category $(\C, \D)$ which models $\Sigma$-types and $\Id$-types. The class $\overline \D$ contains all isomorphisms, is closed under composition and retracts, and is stable under pullback.
\end{lem}

\begin{prop}\label{retractclosure}
Consider a display map category $(\C, \D)$ which models $\Sigma$-types and $\Id$-types. The class $\overline \D$ is the retract-closure (in $\C^\two$, the category of morphisms of $\C$) of $\D$. Moreover, every morphism $f: X \to Y$ of $\overline \D$ is a retract in $\C / Y$ of the display map $\rho(f)$ defined in Proposition \ref{idwfs}.
\end{prop}
\begin{proof}
That $\overline \D$ is the retract closure of $\D$ follows from Lemma \ref{basicliftingproperties}.

Consider any morphism $f: X \to Y$ of $\overline \D$. That $f$ is a retract of $\rho(f)$ follows from Lemma 1.1.9 of \cite{Hov99}, the argument of which we recount here. Consider the following lifting problem.
\[ \diagram 
X \ar@{=}[r] \ar[d]_{\lambda(f)}& X \ar[d]^f\\
X \times_Y \Id(Y) \ar[r]_-{\rho(f)} \ar@{-->}[ur] & Y
\enddiagram \]
It has a solution $s: X \times_Y \Id(Y)  \to X$ since $\lambda(f) \in {^\boxslash \D}$ and $f \in (^\boxslash \D)^\boxslash$.
Then we can rearrange the lifting problem diagram into the following commutative diagram where $f$ appears as a retract of $\rho(f)$ in $\C / Y$.
\[ \diagram 
X \ar@{=}@/^2ex/@<0.5ex>[rr] \ar[dr]_f \ar[r]_>>>{\lambda(f)} & X \times_Y \Id(Y) \ar[r]_-s \ar[d]^{\rho(f)} & X \ar[dl]^f \\
& Y
\enddiagram \qedhere \] 
\end{proof}

\section{Cauchy complete categories}
\label{sec:ccc}
In this section, we prove our main theorem: if a category $\C$ is Cauchy complete and $(\C, \D)$ is a display map category which models $\Sigma$-types and $\Id$-types, then $(\C, \overline \D)$ is also a display map category which models $\Sigma$-types and $\Id$-types. Moreover, if $(\C, \D)$ also models $\Pi$-types, then $(\C, \overline \D)$ models $\Pi$-types as well.

The proofs in this section use the following idea. We need to prove that a certain functor, built out of elements of $\overline \D$, is representable while we hypothesize that the same functor, if built only out of elements of $ \D$, is representable. In a Cauchy complete category, retracts of representable functors are themselves representable (Lemma \ref{retractrepresentable}). Thus, using the fact that every element of $\overline \D$ is a retract of an element of $ \D$, we aim to show that those functors we want to be representable are retracts of functors we know to be representable. 
 
 \subsection{Preliminaries}
 
 In this section, we recall the basic definitions and results that are necessary for our narrative. 
 
\begin{defn}[{\cite[Def.~6.5.1,3,8]{Bor94}}]
A morphism $e: C \to C$ in a category $\C$ is an \emph{idempotent} if $e^2 = e$. A \emph{splitting} of such an idempotent $e$ is a retract of $C$
\[ \diagram
R \ar[r]^i & C \ar[r]^r & R
\enddiagram \]	
such that $ir = e$, and we say an idempotent \emph{splits} if it has a splitting.
The category $\C$ is \emph{Cauchy complete} if every idempotent splits.
\end{defn}

Every splitting of an idempotent $e:C \to C$ arises as a coequalizer (and also an equalizer). We will make extensive use of the following corollaries of this fact so we record them here.

\begin{prop}[{\cite[Prop.~6.5.4]{Bor94}}]
Consider an idempotent $e: C \to C$ in a category $\C$. If $e$ splits as $ir = e$, then $r$ is the coequalizer of the diagram $e,1_C: C \rightrightarrows C$. Conversely, any coequalizer of the diagram $e,1_C: C \rightrightarrows C$ gives a splitting of $e$.
\end{prop}

\begin{cor}\label{splitsquare}
Consider a category $\C$, idempotents $e: C \to C$ and $f: D \to D$ in $\C$, and a morphism $c: C \to D$ making the following diagram commute.
\[ \diagram 
C \ar[rr]^e \ar[d]^c & & C \ar[d]^c \\
D \ar[rr]^f & & D
\enddiagram \]
Then splittings of both $e: C \to C$ and $f: D \to D$ extend uniquely to a splitting of the idempotent $\langle e, f \rangle$ in $\C^\two$. In particular, given a splitting $R \xrightarrow{i} C \xrightarrow{r} R$ of $e$ and a splitting $S \xrightarrow{j} D \xrightarrow{s} S$ of $f$, the following diagram displays the unique splitting of $\langle e, f \rangle$ in $\C^\two$.
\[ \diagram
R  \ar[r]^i \ar@{-->}[d]_{sci} & C \ar[r]^r \ar[d]_c  & R \ar@{-->}[d]^{sci} \\
S \ar[r]^j & D \ar[r]^s & S
\enddiagram \]
Moreover, if $c$ is an isomorphism, then so is $sci$.
\end{cor}
\begin{cor}\label{extendsplittigs}
If $\C$ is Cauchy complete, then $\C^\two$ is Cauchy complete.
\end{cor}
\begin{cor}\label{ccslice}
If $\C$ is Cauchy complete, then any slice $\C/Y$ of $\C$ is Cauchy complete.
\end{cor}

\begin{cor}\label{splittingsunique}
Splittings of idempotents are unique up to unique isomorphism.
\end{cor}

The following lemma will be our main tool in establishing the results of this section.

\begin{lem}[{\cite[Lem.~6.5.6]{Bor94}}] \label{retractrepresentable}
If $\C$ is Cauchy complete, then any retract of any representable functor $\C^\op \to \Set$ is representable.
\end{lem}

\subsection{Display map categories}

\begin{prop}\label{cauchydisplay}
Consider a Cauchy complete category $\C$.
If $(\C, \D)$ is a display map category, then $(\C, \overline \D)$ is one as well.
\end{prop}

\begin{proof}
Since $\D \subseteq \overline \D$ and $\D$ contains all isomorphisms and morphisms to the terminal object, then $\overline \D$ does as well. Since $\overline \D$ is the right class of a lifting pair, it is stable under pullback (Lemma \ref{basicliftingproperties}). It only remains to show that pullbacks of morphisms of $\overline \D$ exist.

Consider a morphism $d: X \to Y$ of $\overline \D$ and a morphism $\alpha: A \to Y$ of $\C$. By Proposition \ref{retractclosure}, $d$ is a retract in $\C / Y$ of some $d': X' \to Y$ in $\D$. Let $P$ denote the pullback diagram category, and let $D, D': P \to \C$ denote the following two pullback diagrams in $\C$.
\[ \diagram 
& X \ar[d]^{d} & & &  X' \ar[d]^{d'}\\
A \ar[r]^\alpha & Y & & A \ar[r]^\alpha  & Y
\enddiagram \]

Let $c$ denote the functor $\C \to [P, \C]$ which sends an object $m$ of $\C$ to the constant functor $c_m : P \to \C$ at $m$.

Then since $d$ is a retract of $d'$ in $\C / Y$, the functor $D$ is a retract of $D'$ in $[P, \C]$, and thus the functor $\Nat(c(-),D): \C \to \Set$ is a retract of $\Nat(c(-),D'): \C \to \Set$. Now since we assume that there is a limit of the pullback diagram $D'$, the functor $\Nat(c(-),D')$ is representable. Therefore, by Lemma \ref{retractrepresentable}, the functor $\Nat(c(-),D)$ is also representable, and we conclude that $D$ has a limit.

Therefore, assuming that pullbacks of morphisms of $\D$ exist, pullbacks of morphisms of $\overline \D$ exist.
\end{proof}

\subsection{$\Sigma$-types}

Since $\overline \D$ is closed under composition, we immediately find that $(\C, \overline \D)$ models $\Sigma$-types. Note that for this result, we only use the hypothesis that $\C$ is Cauchy complete to ensure, by Proposition \ref{cauchydisplay}, that $(\C, \overline \D)$ is a display map category.

\begin{prop}\label{sigmacauchy}
Consider a Cauchy complete category $\C$ and a display map category $(\C, \D)$ which models $\Sigma$- and $\Id$-types. Then $(\C, \overline \D)$ is a display map category which models $\Sigma$-types.
\end{prop}
\begin{proof}
$\overline \D$ is closed under composition by Lemma \ref{basicliftingproperties}, and this means that $(\C, \overline \D)$ models $\Sigma$-types.
\end{proof}
\subsection{$\Id$-types}\label{sec:idtypes}

\begin{prop}\label{funidcauchy}
Consider a Cauchy complete category $\C$.
Suppose that $(\C, \D)$ is a display map category which models $\Sigma$-types and functorial $\Id$-types.
Then $(\C, \overline \D)$ is a display map category which models functorial $\Id$-types.
\end{prop}

\begin{proof}
Fix a slice $\C/ Y$ and an object $e \in \overline \D$ in this slice. We want to construct an $\Id$-type for $e$. There is a $d \in \D$ such that $e$ is a retract of $d$ (Proposition \ref{retractclosure}). Since we have an $\Id$-type on $d$, we have the following diagram in $\C/ Y$ (where $i$, $s$ form the retraction and $r_d, \epsilon_d$ form the $\Id$-type on $d$).
\[
\diagram
e \ar[d]^i & & e \times e \ar[d]^{i \times i} \\
d \ar[d]^-s \ar[r]^{r_d}& \iota_d \ar[r]^-{\epsilon_d} & d \times d \ar[d]^{s\times s}\\
e& & e \times e
\enddiagram \tag{$*$}
\]
Since the $\Id$ type is functorial, there is a morphism $\iota_{\langle is , is \times is \rangle}: \iota_d \to \iota_d$ making the following diagram commute. 
\[
\diagram
d \ar[d]^-s \ar[r]^{r_d}& \iota_d \ar[r]^-{\epsilon_d} \ar[dd]|{\iota_{\langle is , is \times is \rangle}}& d \times d \ar[d]^{s \times s}\\
e \ar[d]^i & & e \times e\ar[d]^{i \times i} \\
d \ar[r]^{r_d}& \iota_d \ar[r]^-{\epsilon_d} & d \times d
\enddiagram 
\]

Since $\langle is , is \times is \rangle$ is an idempotent and this factorization is given functorially, the morphism $\iota_{\langle is , is \times is \rangle}$ is also an idempotent. By Lemma \ref{ccslice}, $\C/Y$ is Cauchy complete, so we can split the idempotent $\iota_{\langle is , is \times is \rangle}$. Then by Corollary \ref{splitsquare}, this extends to splittings of the rectangles in diagram ($*$) above. This gives us the following commutative diagram.
\[
\diagram
e \ar[d]^i  \ar[r]^{r_e}& \iota_e\ar[r]^-{\epsilon_e} \ar[d]^{\iota_i}& e \times e \ar[d]^{i \times i} \\
d \ar[d]^-s \ar[r]^{r_d}& \iota_d \ar[r]^-{\epsilon_d} \ar[d]^{\iota_r}& d \times d \ar[d]^{s \times s}\\
e\ar[r]^{r_e}& \iota_e \ar[r]^-{\epsilon_e} & e \times e
\enddiagram 
\]

Now we see that the morphism $\epsilon_e$ is in $\overline \D$ since it is a retract of $\epsilon_d \in \D$.

Now, we need to show that for any $\alpha: a \to e$, the pullback $\alpha^* r_e$ is in $^{\boxslash}\overline \D$. Let $\epsilon_{xi}$ denote the composition $\pi_i \epsilon_{x}$ for $x = d,e$ and $i = 0,1$. Since $r_e$ is a retract of $r_d$, as shown in the following diagram, $\alpha^* r_e$ is a retract of $\alpha^* r_d$.
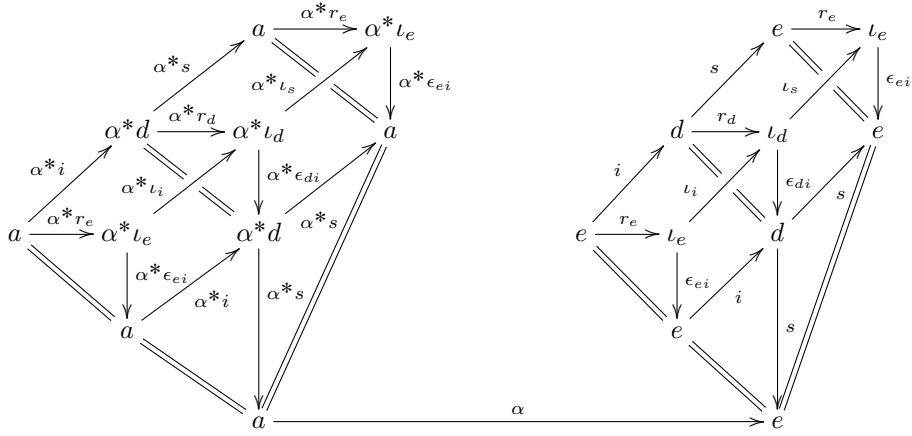
\[ \diagram
&&a \ar[r]^{\alpha^*r_e} \ar@{=}[dr]|\hole &\alpha^* \iota_e \ar[d]^{\alpha^*\epsilon_{ei}}&& & & e \ar[r]^{r_e} \ar@{=}[dr]|\hole & \iota_e \ar[d]^{\epsilon_{ei}}\\
&\alpha^*d \ar[ur]^{\alpha^*s} \ar[r]^{\alpha^*r_d} \ar@{=}[dr]|\hole & \alpha^*\iota_d \ar[ur]^(0.3){\alpha^* \iota_s} \ar[d]^(0.4){\alpha^*\epsilon_{di}} & a \ar@{=}[dddl]&& & d \ar[ur]^s \ar[r]^{r_d} \ar@{=}[dr]|\hole & \iota_d \ar[ur]^(0.3){\iota_s} \ar[d]^{\epsilon_{di}} & e \ar@{=}[dddl] \\
a \ar[ur]^{\alpha^*i} \ar[r]^{\alpha^*r_e} \ar@{=}[dr]& \alpha^*\iota_e \ar[ur]^(0.3){\alpha^*\iota_i} \ar[d]^(0.4){\alpha^*\epsilon_{ei}} & \alpha^*d \ar[dd]^(0.3){\alpha^*s}\ar[ur]_(0.3){\alpha^*s}&& & e \ar[ur]^i \ar[r]^{r_e} \ar@{=}[dr]& \iota_e \ar[ur]^(0.3){\iota_i} \ar[d]^{\epsilon_{ei}} & d \ar[dd]^s\ar[ur]_s \\
&a \ar[ur]_{\alpha^*i}\ar@{=}[dr]&&&& & e \ar@{=}[dr] \ar[ur]_i \\
&&a \ar[rrrrr]^\alpha &&&&& e
\enddiagram \]
Since $\alpha^* r_d$ is in $^{\boxslash} \D$ by hypothesis, and $^{\boxslash} \D$ is closed under retracts, we find that $\alpha^* r_e$ is in $^{\boxslash}\overline \D$.

Therefore, $(\C, \overline \D)$ models functorial $\Id$-types.
\end{proof}

\subsection{$\Pi$-types}

\begin{prop}\label{picauchy}
Consider a Cauchy complete category
of display maps $(\C, \D)$ which models $\Sigma$-types, $\Id$-types, and $\Pi$-types. Then the display map category $(\C, \overline \D)$ also models $\Pi$-types.

\end{prop}
\begin{proof}
Consider morphisms $f: X \to Y $ and $g: W \to X $ which are both in $\overline \D$. 
We aim to obtain a $\Pi$-type $\Pi_f g$.

Note that because 
\[ \rho (g) \times_Y \Id(Y): (W \times_X \Id(X)) \times_Y \Id(Y) \to X \times_Y \Id(Y)\] is a pullback of $\rho (g)$, it is in $ \D$. 

\[ \diagram 
(W \times_X \Id(X)) \times_Y \Id(Y) \ar[d]_{\rho (g) \times_Y \Id(Y)} \ar[r] \pullback & W \times_X \Id(X)\ar[d]^{\rho(g)}\\
 X \times_Y \Id(Y) \ar[r] \ar[d] \pullback & X \ar[d]^f \\
 \Id (Y) \ar[r]^{\epsilon_0}& Y 
\enddiagram \]

For any morphism $x$, let $M(x)$ denote the middle object of the factorization given in Proposition \ref{idwfs}.
We will also denote the morphism $ \rho (g) \times_Y \Id(Y)$ as
\[ M(\rho g): M(f \circ \rho g) \to Mf\]
when it improves readability. (Note that the domain and codomain are indeed the middle objects of the factorizations of $f \circ \rho g$ and $f$, respectively.)

Since $M(\rho g)$ and $\rho f$ are in $\mathcal D$, we can form the $\Pi$-type $\Pi_{\rho f} M(\rho g )$ with the following bijection for any $y: A \to Y $ in $\C$.
\[ \C  / Y(y, \Pi_{\rho f} M(\rho g )) \cong \C / {Mf}(\rho f^* y, M(\rho g ))\]
This means that $\Pi_{\rho(f)} M( \rho g)$ represents the functor 
\[\C  /  {Mf}(\rho f^* -, M(\rho g)): \C/  Y \to \Set.\]
We now show that $\C  /  X(f^* -, g)$ is a retract of this functor, so by Lemma \ref{retractrepresentable}, it will itself be representable.

Let $i$ denote the natural transformation 
\[\C/ X(f^* -, g) \to \C/  {Mf}(\rho f^* -, M(\rho g))\]
which at a morphism $z: Z \to Y$ in $\C$, takes a morphism $m: f^* z \to  g$ in $\C/ X$ 
\[ \xymatrix@R=8pt{
X \times_Y Z \ar[dd]_-{m } \ar[drr]^{f^* z}
\\
& & X
\\
W  \ar[urr]^{g}
}\]
to the following morphism in $\C/  {Mf}$
\[ \diagram
X \times_{\epsilon_0} \Id(Y) {_{\epsilon_1}\times} Z \ar[ddrrrr]^{\rho f^* z}
 \ar[d]_-{ a \times 1_{\Id Y}\times  1_Z} \\
\Id(X)  {_{(f\epsilon_0 \times f\epsilon_1)}\times_{(\epsilon_0 \times \epsilon_1)}} \Id(Y)  {_{\epsilon_1}\times} Z 
\ar[d]_-{1_{ \Id X} \times 1_{\Id Y} \times m (\epsilon_{1X} \times 1_Z)}
\ar[drrrr]^{\epsilon_{X0} \times 1}
\\
\Id(X)  {_{(f\epsilon_0 \times f\epsilon_1)}\times_{(\epsilon_0 \times \epsilon_1)}}\Id(Y) {_{\epsilon_1}\times} W
\ar[d]_{b \times 1_{\Id Y }} \ar[rrrr]^{\epsilon_{X0} \times 1}
& & & & X \times_Y \Id(Y)
\\
W \times_{\epsilon_0} \Id(X) {_{f \epsilon_1} \times_{\epsilon_0}} \Id(Y) \ar[urrrr]^{M( \rho g)}
\enddiagram \]
where $a$ and $b$ are given by solutions to the following lifting problems.
\[ \diagram 
X \ar[r]^r  \ar[d]_{\lambda(f)} & \Id(X) \ar[d]^{\epsilon_0 \times f \epsilon_1 } & & W \ar[r]^-{\lambda(g)} \ar[d]_{rg \times 1} & W {_g \times_{\epsilon_0}} \Id(X) \ar[d]^{\rho(g)}
\\
X {_f \times_{\epsilon_0}} \Id(Y) \ar[r]^-{1 \times \epsilon_1} \ar@{-->}[ur]^a& X \times Y & & \Id(X) {_{\epsilon_1} \times_g} W \ar[r]^{\epsilon_0} \ar@{-->}[ur]^b &  X
\enddiagram \]
(The morphism $\epsilon_0 \times f \epsilon_1$ is in $\overline \D$ because it is the composition of $\epsilon_0 \times  \epsilon_1: \Id(X) \to X \times X$ with $1 \times f: X \times X \to X \times Y$. The morphism $rg \times 1$ is in $^\boxslash \D$ because it is one of the pullbacks of $r: X \to \Id(X)$ ensured to be in $^\boxslash \D$ by the definition of $\Id$-types.)

Then let $r$ denote the natural transformation 
\[ \C/ {Mf}(\rho f^* -, M(\rho g)) \to \C/ X(f^* -, g) \]
which at a morphism $z: Z \to Y$ in $\C$, takes a morphism $n:  \rho f^* z \to M(\rho g)$ in $ \C/  {Mf}$
\[ \diagram
X \times_{\epsilon_0} \Id(Y) {_{\epsilon_1}\times} Z\ar[dr]^{\rho f^* z} \ar[dd]_-{n}\\
& X \times_{\epsilon_0} \Id(Y) \\
W \times_{\epsilon_0} \Id(X) {_{\epsilon_1}\times_{\epsilon_0} } \Id(Y) \ar[ur]^{M( \rho g)}\\
\enddiagram \]
to the following composition in $\C/  {X}$
\[ \diagram
X \times_Y Z \ar[ddrrrr]^{f^* z}
\ar[d]_-{1_X \times r_Y \times  1_Z }\\
X \times_{\epsilon_0} \Id(Y) {_{\epsilon_1}\times} Z \ar[drrrr]^{\pi_X}
\ar[d]_n \\
W \times_{\epsilon_0} \Id(X) {_{\epsilon_1}\times_{\epsilon_0} } \Id(Y) \ar[d]_c \ar[rrrr]^{\epsilon_{X1}} & & & & X
\\
W  \ar[urrrr]^{g}
\enddiagram \]

where $c$ is a solution to the following lifting problem.
\[ \diagram 
W \ar@{=}[r] \ar[d]^{\lambda(g)} & \ar[d]^g W \\
W \times_{\epsilon_0} \Id(X) \ar[r]^-{\rho(g)} \ar@{-->}[ur]^c & X
\enddiagram \]

Now we claim that 
\[ \C/  X(f^* -, g) \xrightarrow{i} \C/  {Mf}(\rho(f)^* -, M(\rho g)) \xrightarrow{r} \C/  X(f^* -, g) \]
is a retract diagram.
To that end, consider a morphism $m$ of $\C/ X(f^* z, g)$. Then $ri(m)$ is the following composition. 
\[ \diagram
X \times_Y Z  \ar[dddrrrr]^{f^* z}
\ar[d]_-{1_X \times r_Y \times  1_Z }\\
X \times_{\epsilon_0} \Id(Y) {_{\epsilon_1}\times} Z \ar[ddrrrr]^{\pi_X}
 \ar[d]_-{ a \times 1_{\Id Y}\times  1_Z} \\
\Id(X)  {_{(f\epsilon_0 \times f\epsilon_1)}\times_{(\epsilon_0 \times \epsilon_1)}} \Id(Y)  {_{\epsilon_1}\times} Z 
\ar[d]_-{1_{ \Id X} \times 1_{\Id Y} \times m (\epsilon_{1X} \times 1_Z)}
\ar[drrrr]^{\epsilon_{X1}}
\\
\Id(X)  {_{(f\epsilon_0 \times f\epsilon_1)}\times_{(\epsilon_0 \times \epsilon_1)}}\Id(Y) {_{\epsilon_1}\times} W
\ar[d]_{b \times 1_{\Id Y }} \ar[rrrr]^{\epsilon_{X1}}
& & & & X 
\\
W \times_{\epsilon_0} \Id(X) {_{f \epsilon_1} \times_{\epsilon_0}} \Id(Y) \ar[urrrr]^{\epsilon_{X1}} \ar[d]_c \\
W \ar[uurrrr]^{g}
\enddiagram \]

The composition $a\circ (1_X \times r_Y) : X \to \Id(X) $ is $r_X$. Thus, the composite of the first three vertical morphisms in the above diagram is 
\[r_X \times  r_Y \times m: X \times_Y Z \to \Id(X)  {_{(f\epsilon_0 \times f\epsilon_1)}\times_{(\epsilon_0 \times \epsilon_1)}}\Id(Y) {_{\epsilon_1}\times} W .\]
Moreover, 
$b \circ (r_X \times 1_W):  W \to W { \times_{\epsilon_0}} \Id(X)$ 
is $1_W\times r_X$ so the composite of the first four morphisms above is
\[m \times r_X \times r_Y: X \times_Y Z \to W \times_{\epsilon_0} \Id(X) {_{f \epsilon_1} \times_{\epsilon_0}} \Id(Y) .\]
 The composite $c \circ (1_W \times r_X): W \to W$ is the identity, so the vertical composite above is $m$.
Therefore, $ri(m) = m$, and $i$ and $r$ form a retract.

Now by Lemma \ref{retractrepresentable}, we can conclude that $\C /  X(f^* -, g) : \C  / Y \to \Set$ is representable by an object which we will denote by $\Pi_f g$. Furthermore, $\Pi_f g$ is a retract of $\Pi_{\rho f} M(\rho g)$. Since $\Pi_{\rho f} M(\rho g)$ is in $\D$, we can conclude that $\Pi_f g$ is in $\overline \D$, the retract closure of $\D$. Therefore, $(\C, \overline \D)$ does in fact model $\Pi$-types.
\end{proof}

\subsection{Summary}
Putting together Propositions \ref{cauchydisplay}, \ref{sigmacauchy}, \ref{funidcauchy}, and \ref{picauchy}, we get the following theorem. 
\begin{thm}\label{maincauchythm}
Consider a Cauchy complete category $\C$ and a display map category $(\C,  \D)$ which models $\Sigma$-types and functorial $\Id$-types. Then $(\C, \overline \D)$ is again a display map category modeling $\Sigma$- and functorial $\Id$-types. If $(\C,  \D)$ also models $\Pi$-types, then $(\C, \overline \D)$ also models $\Pi$-types.
\end{thm}
\begin{proof}
By Proposition \ref{cauchydisplay}, $(\C, \overline \D)$ is a category with display maps. By Proposition \ref{sigmacauchy}, it models $\Sigma$-types. By Proposition \ref{funidcauchy}, it models functorial $\Id$-types. By Proposition \ref{picauchy}, it models $\Pi$-types.
\end{proof}

\begin{cor}
Consider a weak factorization system $(\mathcal L, \mathcal R)$ on a Cauchy complete category $\C$. The following are equivalent:
\begin{enumerate}
\item There is a subclass $\D \subseteq \mathcal R$ such that $\overline \D = \mathcal R$ and $(\C, \D)$ is a display map category which models $\Sigma$- and functorial $\Id$-types.
\item $(\C, \mathcal R)$ is a display map category which models $\Sigma$- and functorial $\Id$-types.
\end{enumerate}
The following are also equivalent: 
\begin{enumerate}
\item There is a subclass $\D \subseteq \mathcal R$ such that $\overline \D = \mathcal R$ and $(\C, \D)$ is a display map category which models $\Sigma$-, functorial $\Id$-, and $\Pi$-types.
\item $(\C, \mathcal R)$ is a display map category which models $\Sigma$-, functorial $\Id$-, and $\Pi$-types.
\end{enumerate}
\end{cor}

\section{Display map categories reflected in weak factorization systems}

In this last section, we remark that our main theorem (\ref{maincauchythm}) can be phrased more categorically as Theorem \ref{thm:subcat} below. Here, we consider various categories of display map categories on a fixed category $\C$. One might want to consider categories of display map categories with more structure, but we give here a simplified account of the situation to just expose a more categorical version of our result without being encumbered by technicalities.

Let $\C$ be a Cauchy complete category. Let $\mathcal S(\C)$ denote the category whose objects are subclasses $\M$ of morphisms of $\C$ and whose morphisms $\mathcal M \to \mathcal N$ are inclusions $\M \subseteq \N$. Then we can identify the following four full subcategories of $\mathcal S(\C)$:
\begin{itemize}
\item $\mathsf{DMC}_{\Sigma,\Id}(\C)$, the full subcategory of $\mathcal S(\C)$ spanned by those $\M$ such that $(\C,\M)$ is a display map category with $\Sigma$- and functorial $\Id$-types;
\item $\mathsf{DMC}_{\Sigma,\Id,\Pi}(\C)$, the full subcategory of $\mathcal S(\C)$ spanned by those $\M$ such that $(\C,\M)$ is a display map category with $\Sigma$-, functorial $\Id$-, and $\Pi$-types;
\item $\mathsf{WFS}_{\Sigma,\Id}(\C)$, the full subcategory of $\mathsf{DMC}_{\Sigma,\Id}(\C)$ spanned by those $\M$ such that $(^\boxslash \M,\M)$ is a weak factorization system; and
\item $\mathsf{WFS}_{\Sigma,\Id,\Pi}(\C)$, the full subcategory of $\mathsf{DMC}_{\Sigma,\Id,\Pi}(\C)$ spanned by those $\M$ such that $(^\boxslash \M,\M)$ is a weak factorization system.
\end{itemize}

Now we can state our main theorem as the existence of a reflector.

\begin{thm}\label{thm:subcat} Consider a Cauchy complete category $\C$.
The category $\mathsf{WFS}_{\Sigma,\Id}(\C)$ is a reflective subcategory of $\mathsf{DMC}_{\Sigma,\Id}(\C)$, and $\mathsf{WFS}_{\Sigma,\Id,\Pi}(\C)$ is a reflective subcategory of $\mathsf{DMC}_{\Sigma,\Id,\Pi}(\C)$. That is, there are left adjoints $L$ in the diagram below.
\[ \diagram
\mathsf{WFS}_{\Sigma,\Id}(\C) \ar@{^{ (}->}@<-5pt>[r]^\bot & \mathsf{DMC}_{\Sigma,\Id}(\C) \ar@<-6pt>[l]_L \\
\mathsf{WFS}_{\Sigma,\Id,\Pi}(\C) \ar@{^{ (}->}[u] \ar@{^{ (}->}@<-5pt>[r]^\bot & \mathsf{DMC}_{\Sigma,\Id,\Pi}(\C) \ar@<-6pt>[l]_L \ar@{^{ (}->}[u]
\enddiagram \]
\end{thm}
\begin{proof}
Consider the endofunctor on $\mathcal S(\C)$ given on objects by $L(\M):= \overline \M$.
By Theorem \ref{maincauchythm}, this functor can be restricted to functors 
\[L : \mathsf{DMC}_{\Sigma,\Id}(\C) \to \mathsf{WFS}_{\Sigma,\Id}(\C), \] \[L : \mathsf{DMC}_{\Sigma,\Id,\Pi}(\C) \to \mathsf{WFS}_{\Sigma,\Id,\Pi}(\C). \] 
To see that these are the left adjoints shown in the statement, we need to show for any $\D $ in $\mathsf{DMC}_{\Sigma,\Id}(\C)$ and any $\mathcal R$ in $\mathsf{WFS}_{\Sigma,\Id}(\C)$ that 
\[ \hom (L {\D }, \mathcal R) \cong \hom ({  \D }, \mathcal R), \]
or, equivalently, that there is an inclusion $\overline  \D  \subseteq \mathcal R$ if and only if there is an inclusion $ \D  \subseteq \mathcal R$. If $\overline  \D  \subseteq \mathcal R$, then since $\D  \subseteq \overline \D$, we have that $ \D  \subseteq \mathcal R$. If $\D  \subseteq \mathcal R$, then $\overline \D  \subseteq \overline  {\mathcal R}$. Since $\overline {\mathcal R} = \mathcal R$ by Lemma \ref{basicliftingproperties}, we have that $\overline \D  \subseteq  {\mathcal R}$.
\end{proof}

\section{Outlook}

In conclusion, we mention ways in which these results can be extended.

In Section \ref{sec:idtypes}, we showed that if $\C$ is a Cauchy complete category and $(\C, \D)$ is a display map category which models functorial $\Id$-types, then $(\C, \overline \D)$ also models functorial $\Id$-types. We needed the hypothesis that the $\Id$-types were \emph{functorial} in order to use the hypothesis that $\C$ was Cauchy complete. However, this was not strictly necessary. In the next paper in this series, we will develop results which imply the following: if $\C$ is a Cauchy complete category and $(\C, \D)$ is a display map category which models $\Id$-types, then $(\C, \overline \D)$ also models $\Id$-types. This already appears as Proposition 2.5.9 in \cite{Nor17}.

In this paper, we have described a relationship between display map categories and weak factorization systems. We hope to upgrade this to a description of the relationship between comprehension categories and more structured weak factorization systems. In particular, the perspective taken in Theorem \ref{thm:subcat} will be the one appropriate for strengthening our results in that direction. This will build upon work done by Moss in \cite{Mos18} in which he makes clear the relationship between Cauchy completion and comprehension categories.

\subsection*{Acknowledgements} I thank my PhD supervisor Martin Hyland for his guidance and many useful discussions regarding this work. I would also like to thank Benedikt Ahrens and Peter LeFanu Lumsdaine for reading and commenting on drafts of this paper. Thanks also go to the editor, Richard Garner, and the anonymous referees whose insightful comments were very helpful in forming the present paper.
 
\bibliographystyle{alpha}
\bibliography{article}

\begin{thebibliography}{vdBG12}

\bibitem[AW09]{AW09}
S.~Awodey and M.~A. Warren.
\newblock Homotopy theoretic models of identity types.
\newblock {\em Math. Proc. Cambridge Philos. Soc.}, 146(1):45--55, 2009.

\bibitem[Bor94]{Bor94}
F.~Borceux.
\newblock {\em Handbook of categorical algebra}, volume~1.
\newblock Cambridge University Press, Cambridge, UK, 1994.

\bibitem[Emm14]{Emm14}
J.~Emmenegger.
\newblock {A category-theoretic version of the identity type weak factorization
  system}, 2014.
\newblock arxiv:1412.0153.

\bibitem[GG08]{GG08}
N.~Gambino and R.~Garner.
\newblock The identity type weak factorisation system.
\newblock {\em Theoret. Comput. Sci.}, 409(1):94--109, 2008.

\bibitem[Hov99]{Hov99}
M.~Hovey.
\newblock {\em Model categories}.
\newblock American Mathematical Society, Providence, RI, 1999.

\bibitem[Joy17]{Joy17}
A.~Joyal.
\newblock {Notes on Clans and Tribes}, 2017.
\newblock arxiv:1710.10238.

\bibitem[LW15]{LW15}
P.~L. Lumsdaine and M.~A. Warren.
\newblock The local universes model: an overlooked coherence construction for
  dependent type theories.
\newblock {\em ACM Trans. Comput. Log.}, 16(3):23:1--23:31, 2015.

\bibitem[Mos18]{Mos18}
S.~Moss.
\newblock {\em The Dialectica Models of Type Theory}.
\newblock PhD thesis, University of Cambridge, 2018.

\bibitem[MP12]{MP12}
J.~P. May and K.~Ponto.
\newblock {\em More concise algebraic topology}.
\newblock University of Chicago Press, Chicago, 2012.

\bibitem[Nor17]{Nor17}
P.~R. North.
\newblock {\em Type theoretic weak factorization systems}.
\newblock PhD thesis, University of Cambridge, 2017.

\bibitem[NPS00]{NPS00}
B.~Nordstr\"{o}m, K.~Petersson, and J.~M. Smith.
\newblock Handbook of logic in computer science.
\newblock pages 1--32. Oxford University Press, Oxford, UK, 2000.

\bibitem[Shu15]{Shu15}
Michael Shulman.
\newblock Univalence for inverse diagrams and homotopy canonicity.
\newblock {\em Mathematical Structures in Computer Science}, 25(5):1203--1277,
  2015.

\bibitem[Str93]{Str93}
T.~Streicher.
\newblock {\em Investigations into intensional type theory}.
\newblock Habilitationsschrift, Ludwig Maximilian University of Munich, 1993.

\bibitem[Tay99]{Tay99}
P.~Taylor.
\newblock {\em Practical foundations of mathematics}.
\newblock Cambridge University Press, Cambridge, UK, 1999.

\bibitem[vdBG12]{BG12}
B.~van~den Berg and R.~Garner.
\newblock Topological and simplicial models of identity types.
\newblock {\em ACM Trans. Comput. Log.}, 13(1):3:1--3:44, 2012.

\bibitem[War08]{War08}
M.~A. Warren.
\newblock {\em Homotopy theoretic aspects of constructive type theory}.
\newblock PhD thesis, Carnegie Mellon University, 2008.

\end{thebibliography}

\end{document}